\documentclass[leqno]{amsart}
\usepackage{amsmath}
\usepackage{amssymb}
\usepackage{amsthm}
\usepackage{graphicx}
\usepackage{enumerate}
\usepackage[mathscr]{eucal}
\theoremstyle{plain}
\newtheorem{theorem}{Theorem}[section]

\newtheorem{cor}{Corollary}[theorem]

\theoremstyle{definition}
\newtheorem{definition}{Definition}[section]

\newtheorem{example}{Example}[theorem]

\usepackage[pagewise]{lineno}

\begin{document}

\title[A complete characterization of smoothness of linear operators]{A complete characterization of smoothness in the space of bounded linear operators}
\author[Debmalya Sain, Kallol Paul, Arpita Mal and Anubhab Ray  ]{ Debmalya Sain, Kallol Paul, Arpita Mal and Anubhab Ray }

\newcommand{\acr}{\newline\indent}

\address[Sain]{Department of Mathematics\\ Indian Institute of Science\\ Bengaluru 560012\\ Karnataka \\India\\ }
\email{saindebmalya@gmail.com}

\address[Paul]{Department of Mathematics\\ Jadavpur University\\ Kolkata 700032\\ West Bengal\\ INDIA}
\email{kalloldada@gmail.com}

\address[Mal]{Department of Mathematics\\ Jadavpur University\\ Kolkata 700032\\ West Bengal\\ INDIA}
\email{arpitamalju@gmail.com}

\address[Ray]{Department of Mathematics\\ Jadavpur University\\ Kolkata 700032\\ West Bengal\\ INDIA}
\email{anubhab.jumath@gmail.com}

\thanks{The research of Dr. Debmalya Sain is sponsored by Dr. D. S. Kothari Postdoctoral Fellowship, under the mentorship of Professor Gadadhar Misra. Dr. Debmalya Sain feels elated to acknowledge the loving and motivating presence of his childhood friend Dr. Kuntal Mukherjee in every sphere of his life! The research of Prof Kallol Paul  is supported by project MATRICS  of DST, Govt. of India. The research of third author is supported by UGC, Govt. of India. The last author would like to thank DST, Govt. of India, for the financial support in the form of doctoral fellowship.
}

\subjclass[2010]{Primary 46B20, Secondary 47L05}
\keywords{smoothness; bounded linear operators; Birkhoff-James orthogonality; semi-inner-product}

\begin{abstract}
We completely characterize smoothness of bounded linear operators between infinite dimensional real normed linear spaces, probably for the very first time, by applying the concepts of Birkhoff-James orthogonality and semi-inner-products in normed linear spaces. In this context, the key aspect of our study is to consider norming sequences for a bounded linear operator, instead of norm attaining elements.  We also obtain a complete characterization of smoothness of bounded linear operators between real normed linear spaces, when the corresponding norm attainment set non-empty. This illustrates the importance of the norm attainment set in the study of smoothness of bounded linear operators.  Finally, we prove that G\^{a}teaux differentiability and Fr\'{e}chet differentiability are equivalent for compact operators in the space of bounded linear operators between a reflexive Kadets-Klee Banach space and a Fr\'{e}chet differentiable normed linear space, endowed with the usual operator norm. 

\end{abstract}

\maketitle

\section{Introduction.}
The study of smoothness of bounded linear operators between Banach spaces is a classical area of research in the field of geometry of Banach spaces. Several mathematicians including Holub \cite{HO}, Heinrich \cite{HR}, Abatzoglou \cite{A}, Hennefeld \cite{HF}, Kittaneh and Younis \cite{KY}, Rao \cite{R,Rao}, Werner \cite{W} etc., have studied it  on Banach spaces as well as Hilbert spaces. Recently Sain et. al. \cite{SPM} obtained separate necessary as well as sufficient conditions for smoothness of a bounded linear operator defined between infinite-dimensional normed linear spaces. However, to the best of our knowledge, no complete characterization of the same is known till date. The primary purpose of this paper is to give a definitive answer to this intriguing question.\\
Let $\mathbb{X},$ $\mathbb{Y}$ denote normed linear spaces and $\mathbb{H}$ denote a Hilbert space. Throughout the present paper, we consider only real normed linear spaces and real Hilbert spaces. Let $B_{\mathbb{X}}$ and $S_{\mathbb{X}}$ denote the unit ball and the unit sphere of $\mathbb{X}$ respectively, i.e., $B_{\mathbb{X}}=\{x\in \mathbb{X}:\|x\|\leq 1\}$ and $S_{\mathbb{X}}=\{x\in \mathbb{X}:\|x\|= 1\}.$ Let $\mathbb{B}(\mathbb{X},\mathbb{Y})~(\mathbb{K}(\mathbb{X},\mathbb{Y}))$ denote the space of all bounded (compact) linear operators from $\mathbb{X}$ to $\mathbb{Y}$. We write $ \mathbb{B}(\mathbb{X}, \mathbb{Y}) = \mathbb{B}(\mathbb{X}) $ and $ \mathbb{K}(\mathbb{X}, \mathbb{Y}) = \mathbb{K}(\mathbb{X}), $ if $ \mathbb{X}= \mathbb{Y}. $ For $T\in \mathbb{B}(\mathbb{X},\mathbb{Y}),$ let $M_T$ denote the set of all elements of $S_{\mathbb{X}}$ at which $T$ attains its norm, i.e., 
\[M_T=\{x\in S_{\mathbb{X}}:\|Tx\|=\|T\|\}.\] 
We say that a sequence $\{x_n\}\subseteq S_{\mathbb{X}}$ is a norming sequence for $T$ if $\|Tx_n\|\to \|T\|.$
Let $\mathbb{X}^*$ denote the dual space of $\mathbb{X}$. An element $f\in S_{\mathbb{X}^*}$ is said to be a norming linear functional for $x\in {\mathbb{X}}\setminus \{0\}$ if $f(x)=\|x\|.$ Note that, Hahn-Banach theorem ensures the existence of at least one norming linear functional for each $x\in {\mathbb{X}}\setminus \{0\}$. An element $x(\neq 0)$ is said to be a smooth point of $\mathbb{X}$ if $x$ has a unique norming linear functional. A normed linear space $\mathbb{X}$ is said to be smooth if every non-zero point of $\mathbb{X}$ is a smooth point. Let us observe that the notion of smoothness carries over naturally to the space of all bounded linear operators between normed linear spaces, endowed with the usual operator norm. To characterize smoothness of bounded linear operators between normed linear spaces, we use the notions of Birkhoff-James orthogonality \cite{B,J} and semi-inner-product \cite{G, L} on normed linear spaces. An element $x\in \mathbb{X}$ is said to be Birkhoff-James orthogonal to $y\in \mathbb{X},$ written as $x\perp_B y,$ if $\|x+\lambda y\|\geq \|x\| $ for all $\lambda \in \mathbb{R}.$ In this paper, we also use the following notations from \cite{C,Sa,SPM}:\\
For $x,~y\in \mathbb{X}$, we say that $y\in x^+$ if $\|x+\lambda y\|\geq \|x\| $ for all $\lambda \geq 0.$ Similarly, we say that $y\in x^-$ if $\|x+\lambda y\|\geq \|x\| $ for all $\lambda \leq 0.$ \\
Let $x^{\perp}=\{y\in \mathbb{X}:x\perp_B y\}.$\\
For $x,y\in \mathbb{X}$ and $\epsilon \in [0,1),$ $x$ is said to be approximate $ \epsilon $-Birkhoff-James orthogonal to $y$, written as $x\perp_D^{\epsilon}y,$ if $\|x+\lambda y\|\geq \sqrt{1-{\epsilon}^2}\|x\| $ for all $\lambda \in \mathbb{R}.$ \\
For $x,y\in \mathbb{X}$ and $\epsilon \in [0,1),$ we say that $y\in x^{+\epsilon}$ if $\|x+\lambda y\|\geq \sqrt{1-{\epsilon}^2}\|x\| $ for all $\lambda \geq 0.$ Similarly, we say that $y\in x^{-\epsilon}$ if $\|x+\lambda y\|\geq \sqrt{1-{\epsilon}^2}\|x\| $ for all $\lambda \leq 0.$ \\
In \cite{J}, James characterized smoothness of elements of $\mathbb{X}$ in terms of right additivity of Birkhoff-James orthogonality.  In particular, he proved that an element $(0 \neq) x \in \mathbb{X}$ is smooth if and only if for any $y,z\in \mathbb{X},$ $x\perp_B y$ and $x\perp_B z$ implies that $x\perp_B (y+z). $ We next mention the concept of semi-inner-product (s.i.p.) on a normed linear space $\mathbb{X},$ which is integral to our present work.\\
\begin{definition}
Let $\mathbb{X}$ be a normed linear space. A function $[~,~]:\mathbb{X}\times \mathbb{X}\to \mathbb{R}$ is said to be a s.i.p. on $\mathbb{X}$ if and only if for any $\alpha, ~\beta \in \mathbb{R} $ and for any $x,~y,~z\in \mathbb{X},$ it satisfies the following:\\
(i)$[\alpha x+\beta y,z]=\alpha[x,z]+\beta[y,z],$\\
(ii)$[x,x]>0,$ whenever $x\neq 0,$\\
(iii)$|[x,y]|^2\leq [x,x][y,y],$\\
(iv)$[x,\alpha y]=\alpha [x,y].$
\end{definition}
  Giles \cite{G} proved that for every normed linear space, there exists a s.i.p. $[~,~]$ which is compatible with the norm, i.e., $[x,x]=\|x\|^2 $ for all $x\in \mathbb{X}.$ Whenever we talk of a s.i.p. in a normed linear space, we assume that the s.i.p. is compatible with the norm. In a normed linear space, there may exist infinitely many semi-inner-products. However, there exists a unique s.i.p. on a normed linear space if and only if the space is smooth. Moreover, in an inner product space, the unique s.i.p. is the inner product itself.
  
  \smallskip  
	After this introductory part, this paper is demarcated into two sections. The first section deals with smoothness of operator $T$ without any restriction on the norm attainment set $M_T$, i.e., $M_T$ may or may not be empty.  We completely characterize smoothness of bounded linear operators defined between any two normed linear spaces, in terms of operator Birkhoff-James orthogonality and s.i.p. Let us re-emphasize here that although several mathematicians have studied smoothness of a bounded linear operator under certain restrictions on the concerned spaces or on the concerned operator \cite{A,DeK,HR,HF,HO,KY,PSG,R,Rao,W}, this is probably for the very first time that a complete characterization of smoothness in operator spaces, in the general most case, is being presented without any restriction on the spaces or on the operator. 
	
	\smallskip
	
In the next section, we obtain further characterizations of smoothness of a bounded (compact) linear operator, when the corresponding norm attainment set is non-empty. We would like to mention here that in the study of smoothness of bounded linear operators, the role of norm attainment set is of utmost importance. It is well-known that for the smoothness of a bounded linear operator $T$ on a Hilbert space $ \mathbb{H}, $ it is necessary that $  M_T \neq \emptyset, $ in fact, $ M_T = \{ \pm x_0 \} $ for some $ x_0 \in S_{\mathbb{H}}.$ The same fact i.e., $ M_T = \{ \pm x_0 \} $ for some $ x_0 \in S_{\mathbb{X}},$ is necessary as well as sufficient for the smoothness of a  compact linear operator defined on a reflexive smooth Banach space. As it turns out, the study of smoothness of a bounded linear operator is less complicated when the corresponding norm attainment set is non-empty. In particular, we prove that  $T$ is smooth if and only if $T$ attains norm at only one pair of points $\pm x_0$, $Tx_0$ is smooth in $\mathbb{Y}$ and for any $A\in \mathbb{B}(\mathbb{X},\mathbb{Y}),$ $ T \bot_B A \Leftrightarrow Tx_0 \bot_B Ax_0. $ We further prove that if $\mathbb{X}$ is a reflexive Kadets-Klee Banach space and $\mathbb{Y}$ is a normed linear space, then a compact operator $T$ is smooth in $\mathbb{B}(\mathbb{X},\mathbb{Y})$  if and only if  $T$ is smooth in $\mathbb{K}(\mathbb{X},\mathbb{Y})$. Let us recall that a normed linear space $\mathbb{X}$ is said to be Kadets-Klee if whenever $ \{x_n\} $ is a sequence in $ \mathbb{X} $ and $ x \in \mathbb{X} $ is such that $ x_n \stackrel{w}{\rightharpoonup} x $ and $ \lim \limits_{n\rightarrow \infty} \|x_n\| = \|x\|, $  then $ \lim\limits_{n\rightarrow \infty} \|x_n - x\| = 0. $ We also obtain an interesting necessary condition for smoothness of a bounded linear operator $T$ in terms of norming sequence, provided $M_T$ is non-empty. We demonstrate the applicability of our study by discussing two concrete examples in light of the results obtained by us in the present paper.

\smallskip
It is well-known  that smoothness of $ \mathbb{X} $ at a point $ x_0 \in \mathbb{X} \setminus \{0\} $ is equivalent to the G\^{a}teaux differentiability of the norm at $ x_0. $ To be more precise, $ (\mathbb{X}, \|~.~\|) $ is smooth at $ x_0 \in \mathbb{X} \setminus \{0\} $ if and only if the norm is G\^{a}teaux differentiable at $ x_0, $ i.e., $  \lim \limits_{h \rightarrow 0} \frac{\|x_0 + hy\| - \|x_0\|}{h} $ exists for all $ y \in \mathbb{X}. $ There is a stronger notion of differentiability in a normed linear space. We say that the norm of $ \mathbb{X} $ is Fr\'{e}chet differentiable at $ x_0 \in S_{\mathbb{X}} $ if $ \lim \limits_{h \rightarrow 0} \frac{\|x_0 + hy\| - \|x_0\|}{h} $ exists uniformly in $ y \in S_{\mathbb{X}}. $ We say that $ \mathbb{X} $ is G\^{a}teaux (Fr\'{e}chet)
differentiable if the norm of $ \mathbb{X} $ is G\^{a}teaux (Fr\'{e}chet)	at each $ x_0 \in S_{\mathbb{X}}. $ It is easy to see that if the norm is Fr\'{e}chet differentiable at $ x_0 \in S_{\mathbb{X}} $ then the norm is G\^{a}teaux differentiable at $ x_0 \in S_{\mathbb{X}}. $ Moreover, the converse is not true in general. We end this paper with a rather surprising result that if $ \mathbb{X} $ is a reflexive Kadets-Klee Banach space and $ \mathbb{Y} $ is a Fr\'{e}chet differentiable normed linear space then G\^{a}teaux differentiability and Fr\'{e}chet differentiability are equivalent in $ \mathbb{B}(\mathbb{X},\mathbb{Y}), $ for elements in $ \mathbb{K}(\mathbb{X},\mathbb{Y}). $ To be more precise, $ \mathbb{B}(\mathbb{X},\mathbb{Y}) $ with the usual operator norm is G\^{a}teaux differentiable at $ T \in \mathbb{K}(\mathbb{X},\mathbb{Y}), $ with $ \|T\|=1, $ if and only if $ \mathbb{B}(\mathbb{X},\mathbb{Y}) $ with the same norm is Fr\'{e}chet differentiable at $ T. $ 

\section{ Smoothness of  operator $T$ without any restriction on $M_T$}

In the following theorem, we obtain our promised characterization of smoothness of a bounded linear operator defined between any two normed linear spaces.

\begin{theorem}\label{theorem:smooth1}
Let $ \mathbb{X}, \mathbb{Y} $ be any two normed linear spaces and $ T \in \mathbb{B}(\mathbb{X}, \mathbb{Y}), $ with $ T \neq 0. $ Then the following two conditions are equivalent:\\
(i) $ T $ is smooth.\\  
(ii) For any $ A \in \mathbb{B}(\mathbb{X}, \mathbb{Y}), $ $ T \bot_{B} A \Leftrightarrow$  for any norming sequence $ \{ x_n \} $ for $T,$  any subsequential limit of $\{ [Ax_n, Tx_n] \}$ is $ 0, $ for any s.i.p. $ [~,~] $ on $ \mathbb{Y}. $
\end{theorem}

\begin{proof}

At first, suppose that $(ii)$ is true. Let $ T \bot_{B} A $ and $ T \bot_{B} B $ for some $ A, B \in \mathbb{B}(\mathbb{X}, \mathbb{Y}). $ We want to show that $ T \bot_{B} (A+B). $ Let $ \{ x_n \} $ be any norming sequence for $T$ and $ [~,~] $ be any s.i.p on $ \mathbb{Y}. $ Then $ \{[Ax_n, Tx_n]\} $ has a convergent subsequence, say, $ \{[Ax_{n_{i}}, Tx_{n_{i}}] \}. $ So, by the given condition $ [Ax_{n_{i}}, Tx_{n_{i}}] \rightarrow 0. $ Also, $ \{[Bx_{n_{i}}, Tx_{n_{i}}] \}$ has a convergent subsequence, say, $ [Bx_{n_{i_{k}}}, Tx_{n_{i_{k}}}] \rightarrow 0. $ Therefore, $ [(A+B)x_{n_{i_{k}}}, Tx_{n_{i_{k}}}] = [Ax_{n_{i_{k}}}, Tx_{n_{i_{k}}}]+[Bx_{n_{i_{k}}}, Tx_{n_{i_{k}}}] \rightarrow 0. $ Then for any scalar $ \lambda, $
\begin{eqnarray*}
 \|T+\lambda (A+B) \|\|T\| & \geq & \|(T+\lambda (A+B))x_{n_{i_{k}}} \|\|Tx_{n_{i_{k}}}\| \\
                       & \geq & |[ (T+\lambda (A+B))x_{n_{i_{k}}}, Tx_{n_{i_{k}}}]| \\
											 & = &    | [Tx_{n_{i_{k}}},Tx_{n_{i_{k}}}] + \lambda[(A+B)x_{n_{i_{k}}},Tx_{n_{i_{k}}}] | \\
											 & \rightarrow & \|T\|^2.	
\end{eqnarray*}
 Therefore, $ T \bot_{B}(A+B). $ Hence, from \cite[Th.5.1]{J}, we have, $ T $ is smooth.\\
Now, suppose that $(i)$ is true. Let $ A \in \mathbb{B}(\mathbb{X}, \mathbb{Y}) $ and for any norming sequence $ \{ x_n \} $ for $T,$  any subsequential limit of $ \{[Ax_n, Tx_n]\} $ be $ 0, $ for any s.i.p. $ [~,~] $ on $ \mathbb{Y}. $ Let $ [Ax_{n_{i}}, Tx_{n_{i}}] \rightarrow 0 $ for some subsequence $\{ x_{n_{i}}\} $ of $\{x_n\}$. Now, we want to prove $ T \bot_{B} A. $ For any scalar $\lambda,$
\begin{eqnarray*}
 \|T+\lambda A \|\|T\| & \geq & \|(T+\lambda A)x_{n_{i}} \|\|Tx_{n_{i}}\| \\
                       & \geq & |[ (T+\lambda A)x_{n_{i}}, Tx_{n_{i}}]| \\
											 & = &    | [Tx_{n_{i}},Tx_{n_{i}}] + \lambda[Ax_{n_{i}},Tx_{n_{i}}] | \\
											 & \rightarrow & \|T\|^2.	
\end{eqnarray*}
Thus, we have, $ \|T+\lambda A \| \geq \|T\| $ for all scalar $ \lambda. $ So, $ T \bot_{B} A. $\\
Now, let $ T \bot_{B} A. $  Suppose $ \{x_n\}$ is a norming sequence for $T$ and $ [~,~] $ is any s.i.p. on $ \mathbb{Y}. $ Since $\{ [Ax_n, Tx_n]\} $ is a bounded sequence of real numbers, it has at least one convergent subsequence. We want to prove that every subsequential limit of $\{ [Ax_n, Tx_n] \}$ is $0. $  On the contrary, suppose that there exists a subsequence $ \{x_{n_{i}}\} $ such that $ [Ax_{n_{i}}, Tx_{n_{i}}] \rightarrow r \neq 0. $ Let us consider the operator $ B = T - \frac{\|T\|^2}{r}A. $ Clearly, $ B \in \mathbb{B}(\mathbb{X}, \mathbb{Y}). $ Now,
\begin{eqnarray*}
 [Bx_{n_{i}}, Tx_{n_{i}}] & = & [(T - \frac{\|T\|^2}{r}A)x_{n_{i}}, Tx_{n_{i}}] \\
                       & = & \|Tx_{n_{i}}\|^2- \frac{\|T\|^2}{r}[Ax_{n_{i}}, Tx_{n_{i}}]\\
											 & \rightarrow & \|T\|^2 - \frac{\|T\|^2}{r}r = 0.	
\end{eqnarray*}
Thus, we have $ [Bx_{n_{i}}, Tx_{n_{i}}] \rightarrow 0. $ Now, proceeding as before, we can show that $ T \bot_{B} B. $ Again, as $  T \bot_{B} A, $ $ T \bot_{B}\frac{\|T\|^2}{r}A . $ Therefore, $ T \bot_{B} B+\frac{\|T\|^2}{r}A, $ i.e., $T\perp_B T,$ as $ T $ is smooth. This clearly gives a contradiction, since $ T \neq 0. $ Thus, every subsequential limit of $ \{[Ax_n, Tx_n] \}$ is $0,$  in any s.i.p. $[~,~]$ on $\mathbb{Y}$.  This completes the proof of the theorem.
\end{proof}

\section{Smoothness of operator $T$ when $M_T \neq \emptyset$}
We first note that if $T$ is a compact linear operator from a reflexive Banach space to any normed linear space, then Paul et. al. \cite{PSG} proved that $T$ is smooth if and only if $M_T = \{\pm  x_0\}, $ for some $x_0 \in S_{\mathbb{X}}$ and $ Tx_0$ is smooth. Combining this result with Theorem \ref{theorem:smooth1}, we get the following theorem.

\begin{theorem}
Let $ 0 \neq T \in \mathbb{K}(\mathbb{X},\mathbb{Y}), $ where $ \mathbb{X} $ is a reflexive Banach space and $ \mathbb{Y} $ is any normed linear space.  Then the following conditions are equivalent:\\
(i) $T$ is smooth.\\
(ii) $ M_T=\{ \pm x_0\} $ for some $ x_0 \in S_\mathbb{X} $ and $ Tx_0 $ is a smooth point in $ \mathbb{Y}. $ \\
(iii) For any $ A \in \mathbb{K}(\mathbb{X}, \mathbb{Y}), $ $ T \bot_{B} A \Leftrightarrow $ for any norming sequence $ \{ x_n \} $ for $T,$  any subsequential limit of $ \{[Ax_n, Tx_n]\} $ is $ 0, $ for any s.i.p. $ [~,~] $ on $ \mathbb{Y}. $
\end{theorem}

Abatzoglou \cite{A} and Paul et. al. \cite{PSG} independently proved that if $T$ is a bounded linear operator on Hilbert space $\mathbb{H}$, then $T$ is smooth if and only if $ M_{T}= \{\pm x_{0}\} $ for some $x_0\in S_{\mathbb{H}}$ and $ \| T\|_{H_{0}} < \|T\|, $ where $ x_{0} \bot H_{0}. $ Combining this result with Theorem \ref{theorem:smooth1}, we obtain the following theorem:

\begin{theorem}
Let $ \mathbb{H} $ be a Hilbert space and $ 0 \neq T \in \mathbb{B}(\mathbb{H}). $ Then the following are equivalent:\\
(i) $T$ is smooth.\\
(ii) $ M_{T}= \{\pm x_{0}\}, $ for some $x_0\in S_{\mathbb{H}}$  and $ \| T\|_{H_{0}} < \|T\|, $ where $ x_{0} \bot H_{0}. $ \\
(iii) For any $ A \in \mathbb{B}(\mathbb{H}), $ $ T \bot_{B} A \Leftrightarrow$ for any norming sequence $ \{x_{n}\} $ for $T,$ any subsequential limit of $\{ \langle Ax_{n}, Tx_{n}\rangle \}$ is $ 0. $
\end{theorem}
These two theorems highlight the role of norm attainment set in the study of smoothness of bounded linear operators.
 In the following theorem, we obtain a  nice characterization of smoothness of $T\in \mathbb{B}(\mathbb{X}, \mathbb{Y})$ for any two normed linear spaces $\mathbb{X}$ and $\mathbb{Y}$ whenever $M_T\neq \emptyset.$ 

\begin{theorem}\label{theorem:nonempty}
Let $ \mathbb{X}, \mathbb{Y} $ be two normed linear spaces. Let $ T \in \mathbb{B}(\mathbb{X}, \mathbb{Y}) $ be such that $ T \neq 0 $ and $ M_T \neq \phi. $ Then $ T $ is smooth operator if and only if the following conditions hold:\\
(i) $ M_T=\{\pm x_0\}, $ for some $ x_0 \in S_{\mathbb{X}}. $\\
(ii) $ Tx_0 $ is smooth point in $ \mathbb{Y}. $\\
(iii) For any $A\in \mathbb{B}(\mathbb{X},\mathbb{Y}),$ $ T \bot_B A \Leftrightarrow Tx_0 \bot_B Ax_0. $
\end{theorem}

\begin{proof}
We first prove the necessary part of the theorem. Let $ T \in \mathbb{B}(\mathbb{X}, \mathbb{Y}) $ be smooth. Since $ M_T \neq \phi, ~(i)$ follows  from \cite[Th. 4.5]{PSG}. \\
 $(ii)$ If possible, suppose that $Tx_0$ is not a smooth point of $\mathbb{Y}$. Then there exist $f,~g\in \mathbb{Y}^*$ such that $f\neq g$, $\|f\|=\|g\|=\|Tx_0\|$ and $f(Tx_0)=g(Tx_0)=\|Tx_0\|^2$. Now, for each $v(\neq \frac{Tx_0}{\|Tx_0\|})\in S_{\mathbb{Y}}$, choose exactly one $ f_v \in S_{\mathbb{Y}^{*}} $ such that $ f_v(v)=\|v\|=1$ and for $\frac{Tx_0}{\|Tx_0\|}$, choose $f_\frac{Tx_0}{\|Tx_0\|}=\frac{1}{\|Tx_0\|}f$. For $\lambda \in \mathbb{R},$ choose $f_{\lambda v}=\lambda f_v $ for all $v\in S_{\mathbb{Y}}$. Clearly, $f_{Tx_0}=f.$ Then the mapping $ [~,~]: \mathbb{Y}\times \mathbb{Y}\longrightarrow \mathbb{R} $ defined by  $ [u,v] = f_{v}(u) $ for all $ u,~v \in \mathbb{Y}, $ is a s.i.p. on $\mathbb{Y}$. Moreover, $f(y)=[y,Tx_0] $ for all $y\in \mathbb{Y}$. Since $f\neq g$, it is easy to check that $\ker(f)\neq \ker(g)$. Therefore, there exists $y\in \ker(g)\setminus \ker(f)$. Hence, $[y,Tx_0]=f(y)\neq 0$. Since $y\in \ker(g)$, we have for each scalar $\lambda,$ 
$$\|Tx_0\|\|Tx_0+\lambda y\|=\|g\|\|Tx_0+\lambda y\|\geq |g(Tx_0+\lambda y)|=\|Tx_0\|^2.$$
 This implies that $Tx_0\bot_B y$. Now, let $H$ be a hyperspace such that $x_0\bot_B H$. Then any $z\in \mathbb{X}$ can be uniquely written as $z=ax_0+h$, where $h\in H$ and $a\in \mathbb{R}$. Define $A:\mathbb{X}\longrightarrow\mathbb{Y}$ by $A(ax_0+h)=ay$. Then it is easy to check that $A\in\mathbb{B}(\mathbb{X},\mathbb{Y})$. Now, for any $\lambda \in \mathbb{R}$ we have, 
$$\|T+\lambda A\|\geq \|Tx_0+\lambda Ax_0\|=\|Tx_0+\lambda y\|\geq \|Tx_0\|=\|T\|,$$ since $Tx_0\bot_B y$ and $x_0\in M_T$. Thus, $T\bot_B A$. Therefore, by Theorem \ref{theorem:smooth1}, we have, $[Ax_0,Tx_0]=0$, since $x_0\in M_T$. This gives that $[y,Tx_0]=f(y)=0$, a contradiction. Hence, $Tx_0$ is smooth. \\
$(iii)$ Suppose that $ A \in \mathbb{B}(\mathbb{X}, \mathbb{Y}). $ If $ Tx_0 \bot_{B} Ax_0, $ then clearly, $ T \bot_{B} A. $ Now, suppose that $ T \bot_{B} A. $ We want to show $ Tx_0 \bot_{B} Ax_0. $ If possible, suppose that $ Tx_0 \not\perp_{B} Ax_0. $ Let $ f $ be the unique norming linear functional of $ Tx_0. $ If possible, suppose that $f(Ax_0)=0.$ Then for any scalar $\lambda,$ $\|Tx_0+\lambda Ax_0\|\geq |f(Tx_0+\lambda Ax_0)|=|f(Tx_0)|=\|Tx_0\|.$ Hence, $Tx_0\perp_B Ax_0,$ a contradiction. Therefore, $ f(Ax_0) \neq 0. $ Let $ f(Ax_0) = r (\neq 0). $ Now, consider the linear operator $ B= T- \frac{\|T\|}{r}A. $ Clearly, $ B \in \mathbb{B}(\mathbb{X}, \mathbb{Y}). $ Then 
\begin{eqnarray*}
f(Bx_0) & = & f(Tx_0- \frac{\|T\|}{r}Ax_0)\\
        & = & f(Tx_0) - \frac{\|T\|}{r}f(Ax_0)\\
				& = & \|Tx_0\|- \frac{\|T\|}{r}r\\
				& = & \|T\|-\|T\|=0.
\end{eqnarray*} 
Therefore, $ Tx_0 \bot_{B} Bx_0, $ by \cite[Th. 2.1]{J}. Now, for all $ \lambda \in \mathbb{R},$ we have,
\[ \|T +\lambda B \| \geq \|Tx_0 + \lambda Bx_0\| \geq \|Tx_0\|= \|T\|. \]
 Hence, $ T \bot_{B} B. $ Now, $ T $ is smooth, $ T \bot_{B} B $ and $ T \bot_{B} \frac{\|T\|}{r}A $ implies that $ T \bot_{B} (B+\frac{\|T\|}{r}A)=T, $ which is a contradiction. Thus, $ T \bot_{B} A \Rightarrow Tx_0 \bot_{B} Ax_0. $
Therefore, we get $ T \bot_{B} A \Leftrightarrow Tx_0 \bot_{B} Ax_0. $\\
Now, we prove the sufficient part of the theorem. Let $ A, B \in \mathbb{B}(\mathbb{X}, \mathbb{Y}) $ be such that $ T \bot_{B} A $ and $ T \bot_{B} B. $ Then from given condition we have, $ Tx_0 \bot_{B} Ax_0 $ and $ Tx_0 \bot_{B} Bx_0. $ Since $ Tx_0 $ is smooth, $ Tx_0 \bot_{B} (Ax_0+Bx_0). $ Thus, we have, $ T \bot_{B} (A+B), $ as $ x_0 \in M_T. $ Therefore, $ T $ is smooth. This completes the proof of the theorem.
\end{proof}

In the following theorem, we obtain an easy sufficient condition for smoothness of $T\in \mathbb{B}(\mathbb{X},\mathbb{Y}).$

\begin{theorem}\label{theorem:suff}
Let $\mathbb{X},\mathbb{Y}$ be normed linear spaces. Let $T\in \mathbb{B}(\mathbb{X},\mathbb{Y})$ be such that  the following conditions hold:\\
(i) $M_T=\{\pm x_0\}$ for some $ x_0 \in S_{\mathbb{X}}. $\\
(ii) $Tx_0$ is a smooth point in $\mathbb{Y}.$\\
(iii) For every norming sequence $\{x_n\}$ for $T,$ $\{x_n\}$ has a convergent subsequence converging to $ ax_0,$ where $|a|=1.$\\
 Then $T$ is smooth. 
\end{theorem}
\begin{proof}
Let $A \in \mathbb{B}(\mathbb{X},\mathbb{Y})$ be such that $T\perp_B A.$ Then from \cite[Th. 2.4]{SPM}, we have, either (a) or (b) holds:\\
(a) There exists a norming sequence $ \{x_{n}\} $ for $T$ such that $ \|Ax_n\| \rightarrow 0. $\\
(b) There exist two norming sequences  $ \{x_{n}\}, ~\{y_{n}\} $ for $T$ and $ ~ \{\epsilon_{n}\}, ~ \{\delta_{n}\} \subseteq [0,1) $ such that $ \epsilon_{n} \rightarrow 0 ,$ $ \delta_{n} \rightarrow 0, $  $ Ax_n \in (Tx_n)^{+\epsilon_{n}}$ and $ Ay_n \in (Ty_n)^{-\delta_{n}} $ for all $ n \in \mathbb{N}. $\\
Suppose $(a)$ holds. Then $Ax_n \to 0.$ Again, by hypothesis, $\{x_n\}$ has a convergent subsequence converging to $ ax_0,$ where $|a|=1.$ Without loss of generality, assume that $x_n\to ax_0.$ Therefore, $Ax_n \to aAx_0.$ So, $Ax_0=0.$ Therefore, $Tx_0\perp_B Ax_0.$\\
Now, suppose that $(b)$ holds. Now, by hypothesis, $\{x_n\}$ and $\{y_n\}$ have  convergent subsequences converging to $ ax_0$ and $bx_0$ respectively, where $|a|=|b|=1.$ Without loss of generality assume that $x_n\to ax_0$ and $y_n\to bx_0.$ Therefore, $ Ax_n \in (Tx_n)^{+\epsilon_{n}}$ gives that for all $\lambda \geq 0,$
\begin{eqnarray*}
\|Tx_n+\lambda Ax_n\| &\geq& \sqrt{1-{\epsilon_n^2}}\|Tx_n\|\\
\Rightarrow \lim \|Tx_n+\lambda Ax_n\| &\geq& \lim\sqrt{1-{\epsilon_n^2}}\lim\|Tx_n\|\\
\Rightarrow \|aTx_0+\lambda aAx_0\| &\geq& \|aTx_0\|\\
\Rightarrow \|Tx_0+\lambda Ax_0\| &\geq& \|Tx_0\|.
\end{eqnarray*} 
Similarly, $ Ay_n \in (Ty_n)^{-\delta_{n}} $ gives that for all $\lambda \leq 0,$ $\|Tx_0+\lambda Ax_0\| \geq\|Tx_0\|.$ Hence, $Tx_0\perp_B Ax_0.$ Thus, for any $A \in \mathbb{B}(\mathbb{X},\mathbb{Y}),$ $T\perp_B A \Leftrightarrow Tx_0 \perp_B Ax_0.$ Therefore, from Theorem \ref{theorem:nonempty}, we have, $T$ is smooth.
\end{proof}

In the following theorem, we show that a compact operator $T$ defined between a reflexive Kadets-Klee Banach space $\mathbb{X}$ and a normed linear space $\mathbb{Y}$ is smooth in $\mathbb{B}(\mathbb{X},\mathbb{Y})$ if and only if $T$ is smooth in $\mathbb{K}(\mathbb{X},\mathbb{Y}).$

\begin{theorem}
	Let $\mathbb{X}$ be a reflexive Kadets-Klee Banach space and $\mathbb{Y}$ be a normed linear space. Then a compact operator $T$ is smooth in $\mathbb{B}(\mathbb{X},\mathbb{Y})$ if and only if $T$ is smooth in $\mathbb{K}(\mathbb{X},\mathbb{Y}).$
\end{theorem}
\begin{proof}
	The necessary part of the theorem is trivial. We only prove the sufficient part of the theorem. Let $T$ be smooth in $\mathbb{K}(\mathbb{X},\mathbb{Y}).$ Then from \cite[Th. 4.2]{PSG}, $M_T=\{\pm x_0\}$ for some $x_0\in S_{\mathbb{X}}$ and $Tx_0$ is a smooth point in $\mathbb{Y}.$ At first, we show that if $\{x_n\}$ is any norming sequence for $T,$ then $\{x_n\}$ has a convergent subsequence converging to $ ax_0,$ where $|a|=1.$ Let $\{x_n\}$ be any norming sequence for $T.$ Since $\mathbb{X}$ is reflexive, $B_{\mathbb{X}}$ is weakly compact and so $\{x_n\}$ has a weakly convergent subsequence  $\{x_{n_k}\}$ which weakly converges to $x$ (say) in $B_{\mathbb{X}}.$ Since $T$ is compact, $Tx_{n_k} \to Tx.$ This implies that $\|Tx_{n_k}\|\to \|Tx\|.$ Thus, $\|Tx\|=\|T\|\Rightarrow x\in M_T=\{\pm x_0\}.$ Hence, $\|x\|=1.$ Therefore, $\|x_{n_k}\|\to \|x\|$. Since $\mathbb{X}$ is a Kadets-Klee Banach space, we obtain $x_{n_k}\to x,$ where $x\in \{\pm x_0\}.$ Thus $x_{n_k} \to ax_0,$ where $|a|=1.$ Thus, $T$ satisfies all the conditions of Theorem \ref{theorem:suff} and so $T$ is smooth in $\mathbb{B}(\mathbb{X},\mathbb{Y}).$   
\end{proof}

In the following theorem, we obtain a necessary condition for smoothness of a bounded linear operator $T$ defined between any two normed linear spaces, in terms of norming sequences, when the norm attainment set is non-empty. We would like to emphasize that the condition obtained by us is of independent interest as it does not concern Birkhoff-James orthogonality.

\begin{theorem}\label{theorem:necessary1}
Let $ \mathbb{X}, \mathbb{Y} $ be any two normed linear spaces. Let $ T \in \mathbb{B}(\mathbb{X},\mathbb{Y}) $ be a smooth operator and  $ M_T \neq \emptyset. $ Then the following conditions hold:\\
 (i) $ M_{T} = \{\pm x_0\} $ for some $ x_0 \in S_{\mathbb{X}} $ \\
(ii)  $ Tx_0 $ is a smooth point in $ \mathbb{Y}. $ \\
(iii)   $ x_0 \in \overline{span \{ x_{n}:  n \in \mathbb{N}\}}, $ whenever $\{x_n\}$ is a norming sequence for $T$.
\end{theorem}
\begin{proof}
Clearly (i) follows from  \cite[Th. 4.5]{PSG} and (ii) follows from  Theorem \ref{theorem:nonempty}. We prove (iii).  Let $\{x_n\}$ be a norming sequence for $T$ and  $ span \{ x_{n}:  n \in \mathbb{N}\}= X_1. $ If possible, let $ x_0 \not\in \overline{X_1}. $ Then we have $ d(x_0, X_1)= r > 0. $ Let $ X_2$ be the subspace of $\mathbb{X}$ spanned by $ \{x_0\} $ and $ X_1 . $ Then define $ f : X_2 \longrightarrow \mathbb{R} $ by $ f(\alpha x_0 + w) = \alpha, $ where $ \alpha \in \mathbb{R} $ and $ w \in X_1. $ Now, for any $\alpha \neq 0, $ $ \|\alpha x_0 + w\| = |\alpha| \| x_0 + \frac{w}{\alpha}\| \geq |\alpha|r. $  Then $ |f(\alpha x_0 + w)| = |\alpha| \leq \frac{\|\alpha x_0 + w\|}{r}. $ So, $ \|f\| \leq \frac{1}{r}. $ Therefore, we have $ f \in X_{2}^{*}. $ Hence, by Hahn-Banach theorem there exists a norm preserving extension $ g \in \mathbb{X}^{*} $ of $ f. $ So, we have $ g(x_0)=1 $ and $ g(w)=0 $ for all $ w \in X_1. $ Now, any element $ z \in \mathbb{X} $ can be uniquely written as $ z= \alpha x_0 + h, $ where $ \alpha \in \mathbb{R} $ and $ h \in \ker(g). $  Define a linear operator $ A : \mathbb{X} \longrightarrow \mathbb{Y} $ by $ A(\alpha x_0+ h) = \alpha Tx_0. $  Then $ \|A(\alpha x_0+h)\|=|\alpha|\|Tx_0\|. $ Now, $ \|g\|\|\alpha x_0 + h \| \geq |g(\alpha x_0 + h)|=|\alpha|. $ Therefore,
\[\|A(\alpha x_0+ h)\|=|\alpha|\|Tx_0\| \leq \|T\| \|g\|\|\alpha x_0 + h \|. \]
Thus, $\|A\| \leq \|T\|\|g\|$ and $ A \in \mathbb{B}(\mathbb{X},\mathbb{Y}). $ Now, $\|T+\lambda A\| \geq \|(T+\lambda A)x_n\|= \|Tx_n\| \rightarrow \|T\| $ for all scalar $ \lambda. $ Therefore, $T \bot_B A $ but $ Tx_0 \not\perp_B Ax_0, $ as $ Tx_0=Ax_0. $ So from Theorem \ref{theorem:nonempty} it follows that $T$ is not smooth, which is a contradiction. Thus, $ x_0 \in \overline{X_1}= \overline{span \{ x_{n}: ~ n \in \mathbb{N}\}}. $ 
\end{proof}

We now give two examples, illustrating the usefulness of Theorem \ref{theorem:necessary1} and Theorem \ref{theorem:suff}, towards identifying smooth bounded linear operators on normed linear spaces in general, and $ l_p $ spaces in particular, provided the corresponding norm attainment set is non-empty.

\begin{example}
Consider a bounded linear operator $T:l_p\to l_p (1<p<\infty),$ defined by $Te_1=e_1$ and $Te_n=(1-\frac{1}{n})e_n ~\forall ~n \geq 2,$ where $e_n=(\underbrace{0,0,\ldots,0,1,}_{n}0,\ldots)$ Then clearly, $ \|T\|=1 $ and $ M_T= \{ \pm e_1\}.$ Also $\|Te_n\|\to 1=\|T\|$ and $e_1 \not \in \overline{span \{e_n:n\in \mathbb{N}\}}.$ Therefore, applying Theorem \ref{theorem:necessary1}, we conclude that $T$ is not smooth.

\end{example}

\begin{example}
Consider a bounded linear operator $T:l_p\to l_p (1<p<\infty),$ defined by $Te_1=e_1$ and $Te_n=\frac{1}{2}e_n ~\forall ~n \geq 2,$ where $e_n=(\underbrace{0,0,\ldots,0,1,}_{n}0,\ldots)$. Then clearly, $\|T\|=1$ and $M_T=\{\pm e_1\}.$ Now, let $\{y_n\}$ be a norming sequence for $T.$ Let $y_n=(a_{1n},a_{2n},a_{3n},\ldots).$ Then $\|Ty_n\|^p=|a_{1n}|^p+\frac{1}{2^p}\sum_{i=2}^{\infty}|a_{in}|^p=|a_{1n}|^p+\frac{1}{2^p}(1-|a_{1n}|^p)=|a_{1n}|^p(1-\frac{1}{2^p})+\frac{1}{2^p}.$ Therefore, $\|Ty_n\|^p\to \|T\|^p\Rightarrow |a_{1n}|^p\to 1.$ Now, $1=\|y_n\|^p=|a_{1n}|^p+\sum_{i=2}^{\infty}|a_{in}|^p$ and $|a_{1n}|^p\to 1$ implies that $\sum_{i=2}^{\infty}|a_{in}|^p\to 0.$ Thus, $\{y_n\}$ has a subsequence converging to $ ae_1,$ where $|a|=1.$ Therefore, from Theorem \ref{theorem:suff}, it follows that $T$ is smooth. 
\end{example}
We end this paper with a rather surprising result that if $ \mathbb{X} $ is a reflexive Kadets-Klee Banach space and $ \mathbb{Y} $ is Fr\'{e}chet differentiable normed linear space then G\^{a}teaux differentiability and Fr\'{e}chet differentiability are equivalent in $ \mathbb{B}(\mathbb{X},\mathbb{Y}), $ for elements in $ \mathbb{K}(\mathbb{X},\mathbb{Y}). $ To be more precise, using a characterization of Fr\'{e}chet differentiability \cite[Th. 3.1]{HR} in the space of bounded linear operators, we prove the following theorem. For this we use the notation $M_T(\delta),$ for $ 0 < \delta < \|T\|$, introduced in \cite{Sain}, which is defined as, $M_T(\delta) = \{ x \in S_{\mathbb{X}} : \|Tx \| > \|T\| - \delta \}$ i.e., $M_T(\delta)$ is the collection of all unit elements $x$ for which the norm of the image is $\delta-$close to $\|T\|.$
\begin{theorem} \label{Theorem:strong-diff}
Let $ \mathbb{X} $ be a reflexive Kadets-Klee Banach space and $ \mathbb{Y} $ be a Fr\'{e}chet differentiable normed linear space. Let $ T $ be a compact operator from $\mathbb{X}$ to $\mathbb{Y}$ with $ \|T\|=1. $ Then $ \mathbb{B}(\mathbb{X},\mathbb{Y}) $ with the usual operator norm is G\^{a}teaux differentiable at $ T $  if and only if $ \mathbb{B}(\mathbb{X},\mathbb{Y}) $ with the same norm is Fr\'{e}chet differentiable at $ T. $  
\end{theorem}
\begin{proof}
``If" part of the theorem is trivial. We here prove the ``only if" part. Let the operator norm on $ \mathbb{B}(\mathbb{X},\mathbb{Y}) $ is G\^{a}teaux differentiable at $ T. $ Then the operator norm on $ \mathbb{K}(\mathbb{X},\mathbb{Y}) $ is G\^{a}teaux differentiable at $ T, $ as $ \mathbb{K}(\mathbb{X},\mathbb{Y}) \subseteq \mathbb{B}(\mathbb{X},\mathbb{Y}). $ To prove the operator norm is Fr\'{e}chet differentiable at $ T, $ we will show that all the conditions of Theorem $ 3.1 $ of \cite{HR} hold true. Since the norm of $ \mathbb{K}(\mathbb{X},\mathbb{Y}) $ is G\^{a}teaux differentiable at $ T, $ $ T $ is smooth. Then from the Theorem $ 4.2 $ of \cite{PSG}, we have $ M_{T}= \{ \pm x_0\} $ for some $ x_0 \in S_{\mathbb{X}} $ and $ Tx_0 $ is a smooth point in $ \mathbb{Y}. $ As $ \mathbb{Y} $ is Fr\'{e}chet differentiable, $ Tx_0 $ is a Fr\'{e}chet differentiable point in $ \mathbb{Y}. $ Thus, conditions (a) and (c) of Theorem $ 3.1 $ of \cite{HR} are satisfied. We will show that (b) also holds true. Let $ \{x_n\} $ be a norming sequence for $ T. $ Let us choose $ \epsilon_0 = \frac{1}{2}. $ Clearly , $ B(x_0, \epsilon_0) \cap B(-x_0, \epsilon_0) = \emptyset. $ By applying Theorem $ 2.3 $ of \cite{Sain}, there exists $ \delta_0 = \delta(\epsilon_0) > 0 $ such that $ M_T(\delta_0) \subseteq B(x_0, \epsilon_0) \cup B(-x_0, \epsilon_0). $ On the other hand, since $ \|Tx_n\| \rightarrow \|T\|, $ there exists $ n_0 \in \mathbb{N} $ such that $ \|Tx_n\| > \|T\|-\delta_0 $ for all $ n \geq n_0. $ In other words, $ x_n \in M_T(\delta_0) \subseteq B(x_0, \epsilon_0) \cup B(-x_0, \epsilon_0) $ for all $ n \geq n_0. $ Since $ B(x_0, \epsilon_0) \cap B(-x_0, \epsilon_0) = \emptyset, $ for any $ n \geq n_0, $ $x_n $ is in exactly one of $ B(x_0, \epsilon_0) $ and $ B(-x_0, \epsilon_0). $ Now, we choose $ t_n = 1 $ for all $ n < n_0. $ Moreover, for any $ n \geq n_0, $ we define $ t_n=1 $ if $ x_n \in  B(x_0, \epsilon_0) $ and $ t_n=-1 $ if $ x_n \in  B(-x_0, \epsilon_0). $ Clearly $ t_n x_n \in B(x_0, \epsilon_0) $ for all $ n \geq n_0. $ Now, we claim that $ t_n x_n \rightarrow x_0. $ Let $ \epsilon > 0 $ be arbitrary. Without loss of generality we may assume that $ \epsilon < \frac{1}{2}=\epsilon_0. $ By applying Theorem $ 2.3 $ of \cite{Sain}, there exists $ \delta = \delta(\epsilon) > 0 $ such that $ M_T(\delta) \subseteq B(x_0, \epsilon) \cup B(-x_0, \epsilon). $ Since $ \|Tx_n\| \rightarrow \|T\|, $ there exists $ n_1 \in \mathbb{N} $ such that $ \|Tx_n\| > \|T\|-\delta, $ i.e., $ x_n \in M_T(\delta) \subseteq B(x_0, \epsilon) \cup B(-x_0, \epsilon) $ for all $ n \geq n_1. $ Let $ n_2 = max\{n_0,n_1\}. $ Then for all $ n \geq n_2, $ by virtue of our choice of $ t_n, $ $ t_n x_n \in B(x_0, \epsilon). $ Thus, $ t_n x_n \rightarrow x_0. $ Therefore, the norm on $ \mathbb{B}(\mathbb{X},\mathbb{Y}) $ is Fr\'{e}chet differentiable at $ T. $ This completes the proof.
\end{proof}

As an immediate application of Theorem \ref{Theorem:strong-diff}, it follows that there exist infinite-dimensional operator spaces, in which G\^{a}teaux differentiability and Fr\'{e}chet differentiability are equivalent.

\begin{cor}
Let $ p,q \in \mathbb{N}, $ with $ 1 < q < p. $ Let $ \mathbb{X} = \mathbb{B}(\ell_p,\ell_q), $ with the usual operator norm. Then G\^{a}teaux differentiability and Fr\'{e}chet differentiability are equivalent for the norm on $ \mathbb{X}. $
\end{cor}

\begin{proof} We first note that $ \ell_p (1 < p < \infty) $ spaces are reflexive, Kadets-Klee and Fr\'{e}chet differentiable Banach spaces. Since $ 1 < q < p, $ it follows from the well-known Pitt's theorem \cite{D} that $ \mathbb{X} = \mathbb{B}(l_p,l_q) = \mathbb{K}(l_p,l_q). $ Therefore, the result follows directly from Theorem \ref{Theorem:strong-diff}.
\end{proof}

\bibliographystyle{amsplain}

\end{document}